\documentclass[12pt]{article}

\usepackage[top=1in,left=0.9in]{geometry}
\usepackage{amsmath,amsthm,amssymb}
\usepackage{enumerate, xcolor}
\usepackage{graphicx}
\usepackage{ifsym}
\usepackage{fancybox}
\usepackage{adjustbox }
\usepackage{lscape}
\usepackage{afterpage}
\usepackage{cancel}
\usepackage{mathtools} 
\usepackage{tikz}
\usepackage{hyperref}
\usepackage{xcolor}

\DeclareMathOperator*{\argmax}{arg\,max}
\usetikzlibrary{calc}
\colorlet{ColorGray}{blue!5}
\colorlet{ColorGreen}{green!70}

\usetikzlibrary{shapes.geometric}
\newcommand{\Stars}[2][fill=yellow,draw=yellow]{\begin{tikzpicture}[baseline=-0.35em,#1]
\foreach \X in {1,...,5}
{\pgfmathsetmacro{\xfill}{min(1,max(1+#2-\X,0))}
\path (\X*1.2em,0) 
node[star,draw,star point height=0.25em,minimum size=1em,inner sep=0pt,
path picture={\fill (path picture bounding box.south west) 
rectangle  ([xshift=\xfill*1em]path picture bounding box.north west);}]{};
}
\end{tikzpicture}}

\newtheorem{thm}{Theorem}
\newtheorem{lem}[thm]{Lemma}
\newtheorem{prop}[thm]{Proposition}
\newtheorem{cor}[thm]{Corollary}

\theoremstyle{definition}

\newtheorem{rem}{Remark}
\newtheorem*{ex}{Example}

\def\lf{\left\lfloor}   
\def\rf{\right\rfloor}
\def\lc{\left\lceil}   
\def\rc{\right\rceil}

\begin{document}

\title{\textsc{No Feedback? No Worries!}\\ The art of guessing the right card}
\author{Tipaluck Krityakierne$^1$ \and Thotsaporn Aek Thanatipanonda$^{2}$}
\date{%
    \footnotesize{$^1$Department of Mathematics, Faculty of Science, Mahidol University, Bangkok, Thailand\\%
    $^2$Science Division, Mahidol University International College, Nakhon Pathom, Thailand}\\
}

\maketitle

\begin{abstract}
In 1998, Ciucu published ``No-feedback card guessing for dovetail shuffles'', an article which gives the optimal guessing strategy for $n$ cards ($n$ even) after $k$ riffle shuffles whenever $k>2\log_{2}\left(n\right)$. We discuss in this article the optimal guessing strategy and the asymptotic (in $n$) expected number of correct guesses for any fixed $k\geq1$. This complements the work achieved two decades ago by Ciucu. \vspace{1em}

\noindent {\bf Keywords:} card guessing; no feedback; optimal strategy; discrete probability; generating function.
\end{abstract}

\section*{The $\cancelto{\bf{math}}{\bf{art}}$ \hspace{1em} of guessing it right}
Over the past several decades, there have been studies and established
theories attempting to tackle many aspects of the \textsc{Card Guessing
Game}, e.g. finding the optimal or worst case strategies, identifying the distribution and the expected number of correct guesses, etc. You might have seen a dealer performing ``riffle shuffles'', perhaps at the Vegas casinos or at the basement poker nights.  Before play can begin, the cards have to be shuffled multiple times to mix the cards randomly.  \cite{BD} proved the famous result that after giving the deck seven riffle shuffles, all the 52 cards will be equally likely to be in one position as in another. 

\vfill
\begin{flushright}
\includegraphics[scale=0.35]{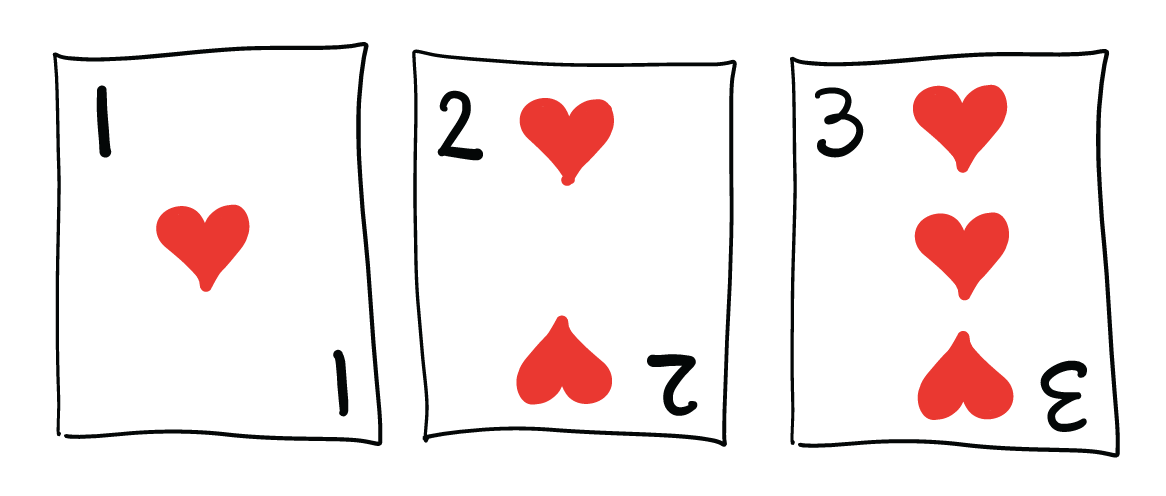}
\end{flushright}

\subsubsection*{\textsc{Card Guessing Game!}}\vspace{-1em}
\begin{flushleft}
\shadowbox{\begin{minipage}[t]{1\columnwidth}%
\begin{footnotesize}
\textsc{Start:} To start, you will need:
\begin{itemize}
\item A dealer and a player
\item 1 deck of $mn$ cards ordered $\underbrace{1,...,1}_{m \text{ copies}}, \dots, \underbrace{n,...,n}_{m \text{ copies}}$
\vspace{1em}
\end{itemize}
\textsc{How to Play?}
\begin{itemize}
\item The dealer shuffles the deck according to some \textsc{Shuffling Rule}
$\mathcal{R}$. 
\item Cards are then dealt from the top, and the player guesses each card
one at a time. 
\item Play continues until the entire deck has been gone through. \vspace{1em}
\end{itemize}
\textsc{Scoring:} The player earns 1 point for each correct guess.
\vspace{1em}

\textsc{Feedback:} The player is given \textsc{Feedback} $\mathcal{F}\in$
\{\textsc{Complete, Partial, No}\} after each guess. \vspace{1em}

\textsc{Object of the Game:} The object of the game is to earn as
many points as possible.%
\end{footnotesize}
\end{minipage}}
\par\end{flushleft}

Let us take a look at a quick guide of how the game is played. As shown above, the player may or may not receive feedback from the dealer depending on the rule. The complete feedback game is similar to blackjack for which the dealer reveals the card after each play (guess). The player thus can adjust their strategy
of play, to improve his odds of guessing the next cards correctly, according to the cards they have already seen. For the no-feedback game, however, as no feedback is given, the player has to guess all the cards beforehand. To increase the probability of guessing the correct number,  it is possible to guess the same card number for cards in different positions. Summary of terminologies and definitions related to the card guessing games are now given.

{\bf{Terminologies} }
\begin{itemize}
\item \textsc{Level of feedback:} The player in the \textsc{No Feedback
}variant is informed nothing about the correctness of his or her guess
after each guess. In the \textsc{Complete} \textsc{Feedback} variant,
on the other hand, the value of the card will be revealed to the player
after each guess. Finally, in the \textsc{Partial Feedback }variant
(also known as \textsc{Yes/No Feedback}), the player is informed whether
his or her guess was correct after each guess, but the value of the
card will not be revealed in case of incorrect guesses. 
\item \textsc{Card shuffling:} If \textsc{Uniform Shuffle} is applied, the
shuffled deck starts in a randomly shuffled state. Applying \textsc{Top
to Random Shuffle} means the dealer takes the top card and places
it back into the deck at position $i$ with probability $p_{i}$.
For \textsc{Riffle Shuffle} (also known as \textsc{GSR Shuffle} or
\textsc{Dovetail Shuffle}), the dealer cuts the deck into two piles
before interleaving the two back into one. Beginning with an ordered
set (1 increasing subsequence), a single riffle shuffle will lead to a permutation
with at most 2 increasing subsequences. The parameter $k$ for the riffle
and top to random shuffles represents the number of times the deck
has been shuffled. In particular, $k=1$ means a single shuffle. 
\item \textsc{Optimal (Best) versus Optimal Mis\`ere Strategy:} A given
strategy is said to be \textsc{Optimal (or Best) Strategy} if the
strategy achieves the maximum expected number of correct guesses,
where the maximum is taken over all strategies. \textsc{Optimal Mis\`ere
(or Worst-case) Strategy,} which achieves the minimum expected number,
is defined analogously \cite{DGS}.
\end{itemize}

We summarize the game variants in the literature,
characterized by the shuffling rules, the level of feedback, and the number
of copies of each card type in Table \ref{tab:relater work}. 

\begin{table}[ht]

{\small{}\caption{Variants of the \textsc{Card Guessing Game} \label{tab:relater work}}
}{\small\par}

{\footnotesize{}}%
\begin{tabular}{|l|l|c|l|l|}
\hline 
{\footnotesize{}Authors} & {\footnotesize{}Shuffling ${\cal R}$} & {\footnotesize{}\#Copies} & {\footnotesize{}Feedback ${\cal F}$} & {\footnotesize{}Contribution}\tabularnewline
\hline 
{\footnotesize{}\cite{DG} } & {\footnotesize{}uniform} & {\footnotesize{}1, $m$} & {\footnotesize{}complete ($m>1$),} & {\footnotesize{}distribution of the number }\tabularnewline
 &  &  & {\footnotesize{}partial ($m=1$)} & {\footnotesize{}of correct guesses,}\tabularnewline
 &  &  &  & {\footnotesize{}optimal and mis\`ere strategies}\tabularnewline
 \hline 
{\footnotesize{}\cite{DGS} } & {\footnotesize{}uniform} & {\footnotesize{}$m$} & {\footnotesize{}partial} & {\footnotesize{}expected number of}\tabularnewline
 {\footnotesize{}and \cite{DGHS} } &  &  &  & {\footnotesize{}correct guesses}\tabularnewline
 \hline 
{\footnotesize{}\cite{BD}} & {\footnotesize{}$k$- riffle } & {\footnotesize{}1} & {\footnotesize{}complete } & {\footnotesize{}a sharp bound on the number }\tabularnewline
 &  &  &  & {\footnotesize{}of shuffles to random deck}\tabularnewline
 \hline 
{\footnotesize{}\cite{C}} & {\footnotesize{}$k$- riffle } & {\footnotesize{}1} & {\footnotesize{}no } & {\footnotesize{}best strategy for $k=1$}\tabularnewline
 &  &  &  & {\footnotesize{}and $k>2\log_{2}\left(n\right)$, $n$ even,}\tabularnewline
 &  &  &  & {\footnotesize{}expected number of correct guesses}\tabularnewline
 \hline 
\textbf{\footnotesize{}\textifsymbol[ifgeo]{100} Present work} & {\footnotesize{}$k$- riffle} & {\footnotesize{}1} & {\footnotesize{}no } & \footnotesize{}best strategy for any fixed $k\geq1$,\tabularnewline
 &  &  &  & {\footnotesize{}expected number of correct guesses}\tabularnewline
 \hline 
{\footnotesize{}\cite{P}} & {\footnotesize{}$k$- top to} & {\footnotesize{}1} & {\footnotesize{}no } & {\footnotesize{}best strategy for $k>4n\log(n)+cn$,}\tabularnewline
 &  {\footnotesize random}&  &  & {\footnotesize{}$n$ even}\tabularnewline
\hline 
{\footnotesize{}\cite{L}} & {\footnotesize{}$1$- riffle } & {\footnotesize{}1} & {\footnotesize{}complete} & {\footnotesize{}expected number of}\tabularnewline
 &  &  &  & {\footnotesize{}correct guesses}\tabularnewline
 \hline 
{\footnotesize{}\cite{KT}} & {\footnotesize{}$1$- riffle } & {\footnotesize{}1} & {\footnotesize{}complete} & {\footnotesize{}(refined) expected number }\tabularnewline
 &  &  &  & {\footnotesize{}of correct guesses, higher moments}\tabularnewline
\hline 
\end{tabular}{\footnotesize\par}
\end{table}
\par


{\bf Our contribution}

As specified in the table, our contribution in this paper builds on and fills gaps of the no-feedback card guessing game with riffle shuffles in the literature. Specifically, \cite{C} proved that, given a deck of $n$ cards ($n$ even), the best no-feedback guessing strategy after
$k$ riffle shuffles when $k>2\log_{2}\left(n\right)$ is to guess
card number 1 for the first half and card number $n$ for the second half of the
deck. The expected number of correct guesses in this case is therefore at most 2.
By using this strategy, Ciucu showed that after shuffling the cards the number of order $\log_{2}\left(n\right)$ times, the deck is well mixed, i.e. close to randomly distributed
deck. In addition, Ciucu also gave the strategy for the other extreme
case, $k=1$. In this work, we focus exclusively on the analysis
of games for a general but finite value of $k\geq1$, which complements work
by \cite{C} in the no-feedback card guessing game with riffle shuffle
variant.

\subsection*{Time to play}

Let us consider now a strategy of guessing the cards after the deck is riffle-shuffled $k$-times, for $k \geq 1$, following these rules:
\begin{itemize}
\item At the beginning, a deck of $n$ cards is ordered $1, 2, 3,\dots, n$.
\item The dealer riffle shuffles the deck $k$ times. 
\item No feedback is given after each guess.
\end{itemize}

{\bf Let's start with $k=1$, a single riffle shuffle}

Take out a deck of $n$ cards, cut it into two piles (possible to have 0 cards in one of the piles), and interleave the cards from two piles in any possible way. What are all possible permutations you can generate? 

\begin{figure}[ht]
\centering
\includegraphics[scale=0.3]{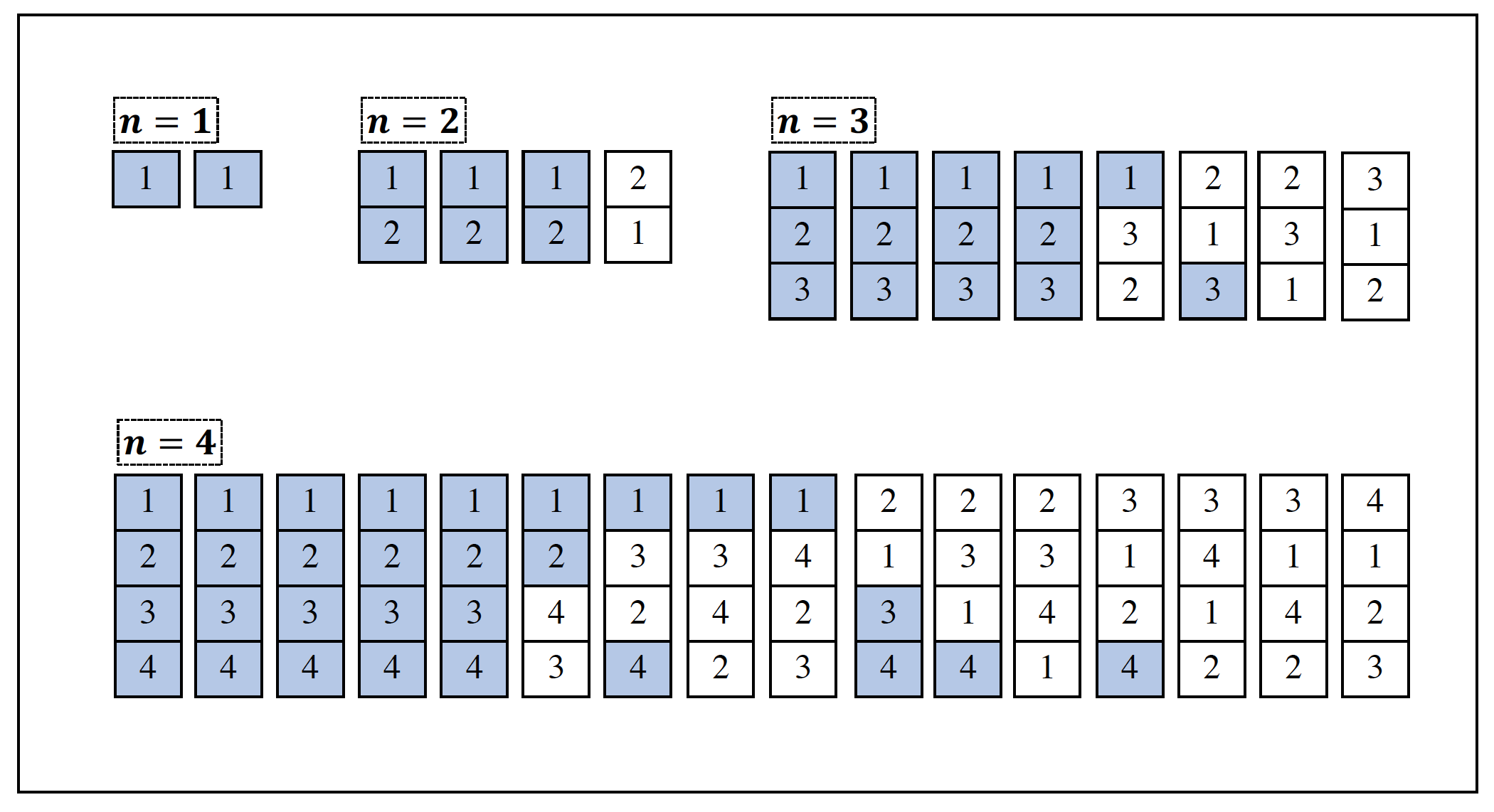}
\caption{\label{fig:permutations}Examples of all possible permutations after 1-time riffle shuffling for $n=1,2,3,4$. For each card position, the color indicates the most likely number.}
\end{figure}

For instance, when $n=1$, $\{1\}$ is the only possible permutation, but with multiplicity 2, generated from the piles $\left(\{1\},\{ \}\right)$ or $\left(\{ \},\{1\}\right)$. In general, there are $2^n$ possibilities, among which $n+1$ are the identity permutation (having one increasing subsequence), and the others are of two increasing subsequences (with multiplicity 1). We give a few examples in Figure \ref{fig:permutations}. For each card position, the color indicates the most likely number to show up in that position.

{\bf What number is most likely to show up in position $i$?}

We record the frequency of  a card in position $i$ being number $j$ in the matrix $M= \left(m_{i,j}\right).$ Then, the row $i$ of the matrix corresponds to position $i$ of the card.  The column $j$ of the matrix corresponds to card number $j$.
For example $m_{1,3}$ gives the number of times the top card is 3.

When $n=4$, the matrix $M$ is given by

\[ M = \left[\begin{array}{cccc}
\fbox{\bf9} & 3 & 3 & 1 \\
4& \fbox{\bf6}& 4& 2 \\
2& 4& \fbox{\bf6}& 4\\ 
1& 3& 3& \fbox{\bf9}\\
\end{array}\right] . \]

{\bf Probability matrix for $k=1$ shuffle, $P^{(1)}$}

For a fixed $n$, we define a probability matrix for $k=1$ shuffle, $P^{(1)}:=M/2^n$. Let $a(i,j,n):= \dfrac{m_{i,j}}{2^n}$ be the $(i,j)$ element of $P^{(1)}$.  Then, $a(i,j,n)$ represents the probability of a card in position $i$ being number $j$. It turns out that $a(i,j,n)$ has an exact formula:

\begin{equation}
\label{eq:a1}
a(i,j,n) = \dfrac{1}{2^i}\binom{i-1}{j-1}+\dfrac{1}{2^{n-i+1}}\binom{n-i}{j-i}. 
\end{equation}

This formula was mentioned and proved earlier in \cite{C}.
We notice that  the first term is non-zero when $i \geq j$ and
the second term is non-zero when $j \geq i.$
The row sum and column sum are both equal to $1,$
i.e. $P^{(1)}$ is a \textit{doubly stochastic matrix}.
This is easy to show in one direction, but not so easy to show in the other direction:
\[ \sum_{j=1}^n a(i,j,n) 
= \dfrac{1}{2^{i}}\sum_{j=1}^i\binom{i-1}{j-1}+\dfrac{1}{2^{n-i+1}}\sum_{j=i}^n\binom{n-i}{j-i}
= \dfrac{1}{2}+\dfrac{1}{2}= 1,  \]
and
\[ \sum_{i=1}^n a(i,j,n) 
= \sum_{i=j}^n \dfrac{1}{2^i}\binom{i-1}{j-1}+\sum_{i=1}^j \dfrac{1}{2^{n-i+1}}\binom{n-i}{j-i}
=1.  \]

Also, the symmetic property is easily observed through the matrix $P^{(1)}$: 
\begin{equation} 
\label{eq:sym}
a(i,j,n) = a(n+1-i,n+1-j,n).
\end{equation}

\textbf{Guessing strategy exists. Do not blind guessing!}

As we now have a formula for the probability that each number $j$ shows up in the position $i$ stored in the matrix $P^{(1)}$, it is possible to find the {\it optimal guessing strategy}
that gives the maximum expected number of correct guesses. 
Since the first row of $P^{(1)}$ corresponds to the top card, the second row the second card, and so on, it is easily shown that one should pick the index of the largest value from each row $i$  as a best guess for the card in the position $i$, for $i=1, \dots, n$. 

\textbf{The expected number of correct guesses}

Let $X^{\cal{G}}$ be a random variable representing the number of correct guesses of an $n$-card deck if the player uses strategy $\cal{G}$. 
To find the expectation $E\left[X^{\cal{G}}\right]$, we consider 
\[
X^{\cal{G}}_{i}=\begin{cases}
\begin{array}{c}
1\\
0
\end{array} & \begin{array}{c}
\text{if the player guesses the }i\text{th position correctly}\\
\text{otherwise.}
\end{array}\end{cases}
\]

Since
\begin{align*}
E\left[X^{\cal{G}}\right]& =\sum_{i=1}^{n}E\left[X^{\cal{G}}_{i}\right]=\sum_{i=1}^{n}P\left(X^{\cal{G}}_{i}=1\right)\leq\sum_{i=1}^{n}\max_{1\leq j_i\leq n} a(i,j_i,n),
\end{align*}
the optimal strategy can be defined as follows.

\vspace{-1em}
\begin{flushleft}
\shadowbox{\begin{minipage}[t]{1\columnwidth}%
{\bf \textsc{Optimal guessing strategy $\cal{G}^*$}}\\
For a fixed row $i$ ($1\leq i\leq n$) of the matrix $P^{(1)}$, let $j_i^*:=\argmax_j a(i,j,n)$.

Under the optimal guessing strategy $\cal{G}^*$, the player guesses number $j_i^*$ for the card position $i$, and in this case, \[E\left[X^{\cal{G}^*}\right]=\sum_{i=1}^{n} a\left(i,j^*_i,n\right).\]
\end{minipage}}
\par\end{flushleft}

For example, when $n=4$, the largest number from each row $i$ of the matrix $M$ (those numbers having a small square around) tells us that  $j_1^*=1$, $j_2^*=2$, $j_3^*=3$, and $j_4^*=4$. Thus, the optimal strategy is to guess number 1, 2, 3, 4, respectively. 

In general, as we shall see later in our main theorem (Theorem \ref{thm:main}, in its simplest case) that for a 1-time shuffled deck, the optimal strategy is to guess the top half of the deck with sequence 
\[ 1,2,2,3,3,4,4,\dots\]
and guess the bottom half in the reverse manner, i.e.
\[ \dots, n-3, n-3, n-2, n-2, n-1, n-1, n. \]

To end the section, we restate a result from \cite{C} regarding the value of $E\left[X^{\cal{G}^*}\right]$ for $k=1$. The main result of our paper is to obtain similar results, but for any $k\geq 1$ times riffle shuffles. 

\begin{thm}[Ciucu, 1998]
\label{thm:Ciucu}
For a deck of $n$ cards, given that the no-feedback guessing of the 
whole deck after 1-time shuffle goes according to the best strategy, 
the average number of correct guesses is $\dfrac{2\sqrt{n}}{\sqrt{\pi}}+
\mathcal{O}(1).$
\end{thm}

\section{Preview: the $k$-time riffle shuffles}
\vspace{-4.5em}
\begin{flushright}
\begin{tikzpicture}
\node [anchor=west] at (1.3,.7) {\bf\scriptsize{0\%}};
\draw [fill=ColorGray] (0,0) rectangle (2,.5);
\draw [fill=ColorGreen] (0,0) rectangle (0,.5);
\end{tikzpicture}
\end{flushright}

We now turn to a general case in which a deck of $n$ cards is given $k>1$ times riffle shuffles. Our goal is to obtain a similar statement to Theorem \ref{thm:Ciucu} for asymptotic $n$. 

{\bf Probability matrix for general $k$ shuffles, $P^{(k)}$}

Analogous to $a(i,j,n)$ for the case $k=1$, we let $a_{i,j}^{(k)}$ denote the probability that the card number $j$ will show up at position $i$ after giving the deck $k$ riffle shuffles (where we now make $i,j$ subscripts, and suppress $n$ from $a_{i,j}^{(k)}$ for simplicity). The probabilty matrix $P^{(k)}$ for $k$ shuffles is subsequently defined by $\left(P^{(k)}\right)_{i,j}=a_{i,j}^{(k)}$.

Since the transpose matrix $\left(P^{(1)}\right)^T$ is a probability transition matrix, it follows that $P^{(k)}=\left(P^{(1)}\right)^k$, the $k$-th power of the matrix, and so

\begin{equation}
\label{eq:a_k}
a_{i,j}^{(k)} =  \sum_{s_1=1}^n \dots 
\sum_{s_{k-1}=1}^n a_{i,s_1}^{(1)}a_{s_1,s_2}^{(1)}\dots a_{s_{k-1},j}^{(1)}. 
\end{equation}

Similar to the case $k=1$, the expected number of correct guesses under the optimal strategy $\cal{G}^*$ is
\begin{equation}
\label{eq: EX}
E\left[X^{\cal{G}^*}\right] = \sum_{i=1}^n a_{i,j_i^*}^{(k)}.
\end{equation}

Thus, finding the optimal guessing strategy boils down to determining the optimal strategy $j_i^*$ which maximizes $a_{i,j}^{(k)}$ for each $i=1,\dots, n$.

{\bf Caution!}
Even though we have a closed-form expression for $a_{i,j}^{(1)}$ as given in Equation \ref{eq:a1}, simply plugging the closed-form formula in Equation \ref{eq:a_k} in attempt to solving for the maximizer of $a_{i,j}^{(k)}$ directly is unfortunately not tractable.

\newpage
Before we dive into the details of the proof, we give an overview of the proof procedure, together with the accompanying Figure \ref{fig:overview} depicting the p.m.f. $a_{60,j}^{(3)}$ (i.e. $i=60, k=3$) for $n=200$ cards. Here, we want to find the maximizer index $j_{60}^*$ of the p.m.f. In practice, the expression of the p.m.f. (the red line) is not tractable. We show later that the p.m.f. is indeed a mixture model with $2^k$ components, whose expressions are tractable via the generating functions.

{\bf \textsc{Proof procedure}}\vspace{-1em}
\begin{flushleft}
\shadowbox{\begin{minipage}[t]{1\columnwidth}%
\begin{small}
\textsc{Setup:}  We fix row $i$ of matrix $P^{(k)}$, and let $J_i$ be a random variable whose probability mass function (p.m.f.) is given by 
\[P(J_i=j)=a_{i,j}^{(k)}, \;\;\; {\text { for }} j=1,\dots, n.
\]

\textsc{Goal:}  Find the index $j_i^*$ that maximizes $P(J_i=j)$.\\

\textsc{Procedure:} The main steps in the proof proceed as follows:
\begin{enumerate}
\item Decompose the p.m.f. of $J_i$ into a mixture model with $2^k$ components: 
\[P(J_i=j)=\frac{1}{2^k}\sum_{t=1}^{2^k}p_{L_t}(j).\]
\item Define, for each $p_{L_t}$, the probability generating function (p.g.f.): 
\[F_{L_t}(x) = \sum_{j=1}^n p_{L_t}(j) \cdot x^j.\]
\item Recover  the distribution $p_{L_t}$ through the p.g.f. $F_{L_t}(x)$.
\item Find, for $n$ large,  the location and the height of the peak of $p_{L_t}$.
\item Show, for $n$ large, that the components $p_{L_1}, \dots, p_{L_{2^k}}$  of the mixture model  are (almost) non-overlapping.
\end{enumerate}

\vspace{1em}
\textsc{Maximizer:} The location of the highest peak among the $2^k$ components is the solution $j_i^*$.
\end{small}
\end{minipage}}
\par\end{flushleft}

\begin{figure}[h]
\centering
\includegraphics[scale=0.33]{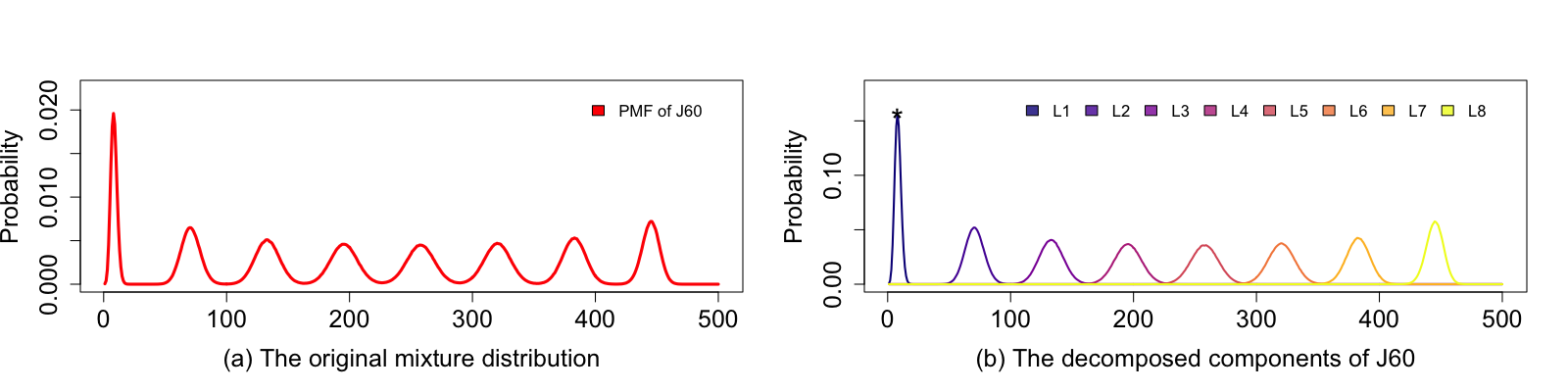}
\caption{\label{fig:overview} Example ($k=3$, $n=200$, $i=60$): (a) the p.m.f. of the mixture model $J_{60}$; \newline (b) the decomposed components of the mixture model,  $P(J_{60}=j)=\frac{1}{8}\sum_{t=1}^{8}p_{L_t}(j)$}
\end{figure}


\section{Game start: the generating function of $p_L$}
\vspace{-4.5em}
\begin{flushright}
\begin{tikzpicture}
\node [anchor=west] at (1.2,.7) {\bf\scriptsize{13\%}};
\draw [fill=ColorGray] (0,0) rectangle (2,.5);
\draw [fill=ColorGreen] (0,0) rectangle (0.25,.5);
\end{tikzpicture}
\end{flushright}

As outlined in the proof procedure, we let $i=1,\dots, n$ be fixed, and focus exclusively on the $i$th row of the matrix $P^{(k)}$ whose  p.m.f. is given by 
\[
P(J_i=j)=a_{i,j}^{(k)},\,\,\, j=1,\dots, n.
\]

We start by decomposing $a_{i,j}$  ($k=1$):
\begin{equation}
\label{eq:a1decomp}
a_{i,j} = \dfrac{1}{2}[b^{(1)}_{i,j} + b^{(2)}_{i,j}],
\end{equation}
where
\[ b^{(1)}_{i,j} :=\dfrac{1}{2^{i-1}}\binom{i-1}{j-1},  \qquad\qquad b^{(2)}_{i,j} := \dfrac{1}{2^{n-i}}\binom{n-i}{j-i}.\]

In general, for $ k \geq 1$, we decompose $a_{i,j}^{(k)}$ 
into the sum of $2^k$ possible products of $b_{i,j}^{(l)}$.
That is, we write 
\begin{equation}  
\label{eq:ak_decompose1}
a_{i,j}^{(k)} = \dfrac{1}{2^k} \sum_L p_L(j),
\end{equation}
where $L$ is the list of length $k$ whose entries are either 1 or 2, 

\begin{equation}
\label{eq:pL_def}
 p_L(j) :=  \sum_{s_{1}=1}^n \dots \sum_{s_{k-1}=1}^n b^{(L[1])}_{i,s_1}
b^{(L[2])}_{s_1,s_2}\dots b^{(L[k])}_{s_{k-1},j} , 
\end{equation}
and $L[l]$ refers to the $l$th element of $L$.

For example, 
for $k =2,$  
\[a_{i,j}^{(2)} = \dfrac{1}{4}\left[p_{[1,1]}(j)+ p_{[2,1]}(j)+ p_{[1,2]}(j)+ p_{[2,2]}(j)\right],\]
where
\begin{align*}
p_{[1,1]}(j) &= \sum_{s=1}^n b^{(1)}_{i,s}b^{(1)}_{s,j} 
= \sum_{s=1}^n \dfrac{1}{2^{i-1}}\binom{i-1}{s-1}\dfrac{1}{2^{s-1}}\binom{s-1}{j-1},  \\
p_{[2,1]}(j) &= \sum_{s=1}^n b^{(2)}_{i,s}b^{(1)}_{s,j} 
=\sum_{s=1}^n \dfrac{1}{2^{n-i}}\binom{n-i}{s-i}\dfrac{1}{2^{s-1}}\binom{s-1}{j-1},  \\
p_{[1,2]}(j) &= \sum_{s=1}^n b^{(1)}_{i,s}b^{(2)}_{s,j} 
=\sum_{s=1}^n \dfrac{1}{2^{i-1}}\binom{i-1}{s-1}\dfrac{1}{2^{n-s}}\binom{n-s}{j-s},  \\
p_{[2,2]}(j) &= \sum_{s=1}^n b^{(2)}_{i,s}b^{(2)}_{s,j} 
=\sum_{s=1}^n \dfrac{1}{2^{n-i}}\binom{n-i}{s-i}\dfrac{1}{2^{n-s}}\binom{n-s}{j-s}.  
\end{align*}

Another example for $k=3$, $L = [1,1,2]$ leads to
\[ p_L(j) =  \sum_{s_{1}=1}^n 
\sum_{s_2=1}^n b^{(1)}_{i,s_1}b^{(1)}_{s_1,s_2}b^{(2)}_{s_{2},j}
=\sum_{s_{1}=1}^n \sum_{s_2=1}^n \dfrac{1}{2^{i-1}}
\binom{i-1}{s_1-1}\dfrac{1}{2^{s_1-1}}\binom{s_1-1}{s_2-1}\dfrac{1}{2^{n-s_2}}\binom{n-s_2}{j-s_2}  . \]

{\bf Generating function of $p_L$}

The expression of $p_L$ in Equation \ref{eq:pL_def} initially involves $(k-1)$-nested binomial sum, for which we have no means of simplifying. The crux of our proposed approach is to find the generating function of $p_L$, giving rise to $k$-nested sum. After reordering the summation by moving the outermost sum (the summation over index $j$) to innermost, and evaluating the summations from the inside (that of the index $j$) to the outside (those of $s_{k-1},\dots, s_1$), fortunately, we obtain a very simple generating function for all $k\geq1$.

We define the generating function by
\begin{equation} 
\label{eq:fL}
F_L(x) = \sum_{j=1}^n p_L(j) \cdot x^j. 
\end{equation}

\begin{rem}
$F_L(x)$ is a valid probability generating function as $F_L(1) = 1.$  This also ensures that each $p_{L}$ in Equation \ref{eq:ak_decompose1} is a valid probability mass function. 
\end{rem}


Let us look at some examples of $F_L(x)$:
\begin{itemize}
\item for $k=0$,
\[ F_{[\,\,]}(x) = x^i.\] \\
This is a trivial case. When cards are not shuffled, the probability that the card in position $i$ is $i$ is 1. 

\item for $k=1$,
\[  F_{[1]}(x) = (0x+1)^{n-i}\left(\dfrac{x}{2}+\dfrac{1}{2}\right)^{i-1} x,
\qquad  F_{[2]}(x) = \left(\dfrac{x}{2}+\dfrac{1}{2}\right)^{n-i}\cdot \left(x+0\right)^{i-1} x.\]

\item for $k=2$, 
\[  F_{[1,1]}(x) =  (0x+1)^{n-i}\left(\dfrac{x}{4}+\dfrac{3}{4}\right)^{i-1} x,\qquad 
F_{[2,1]}(x) = \left(\dfrac{x}{4}+\dfrac{3}{4}\right)^{n-i}\left(\dfrac{x}{2}+\dfrac{1}{2}\right)^{i-1} x,\]
\[ F_{[1,2]}(x) = \left(\dfrac{x}{2}+\dfrac{1}{2}\right)^{n-i}\left(\dfrac{3x}{4}+\dfrac{1}{4}\right)^{i-1} x, \qquad   F_{[2,2]}(x) = \left(\dfrac{3x}{4}+\dfrac{1}{4}\right)^{n-i}\left(x+0\right)^{i-1} x.\]
\end{itemize}

{\bf Hooray! $F_L(x)$ has a nice simple pattern}

It is utterly important that $F_L(x)$ has a simple pattern from which the underlying structure of the problem can be uncovered. To keep track of the pattern, we define a functional notation:
\begin{equation}
\label{eq:g}
g(a,b) = [ax+(1-a)]^{n-i}\cdot[bx+(1-b)]^{i-1}\cdot x.
\end{equation} 

With this notation, the function $F_L(x)$ for $k=0, 1$ can be expressed as
\begin{itemize}
\item for $k=0$,
\[ F_{[\,\,]}(x) = g(0,1).\]

\item for $k=1$,
\[ F_{[1]}(x) = g\left(0,\dfrac{1}{2}\right), \;\ F_{[2]}(x) = g\left(\dfrac{1}{2},1\right). \]

\end{itemize}

\begin{figure}[h]
\noindent
\includegraphics[scale=0.45]{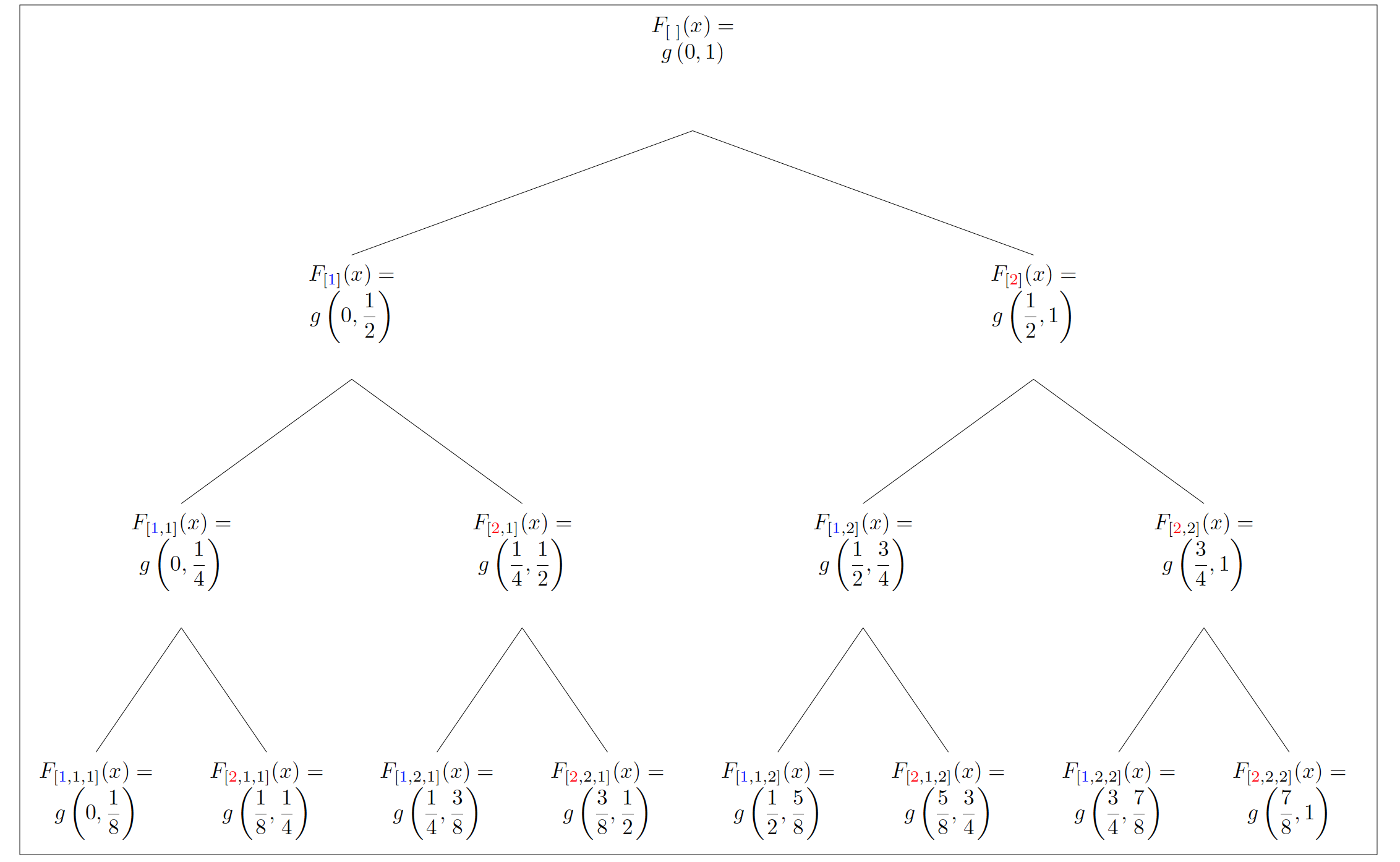}
\caption{\label{fig:tree} A bisection tree structure.  Each node of depth $k$ represents $ F_{L_t}(x)$ for $t=1, 2, \dots, 2^k$. }
\end{figure}

For general $k\geq1$, a bisection tree in Figure \ref{fig:tree} illustrates the connections between the expressions $F_{L}(x)$ of the $k$ and $k+1$ shuffles. 
At any subtree, the argument $(a,b)$ for the function $g(\cdot)$ of the parent node is bisected, yielding the arguments for the left  $(a,\frac{a+b}{2})$  and the right  $(\frac{a+b}{2},b)$ child nodes. 
It is important to note that the parent's index $L_t$ is expanded by adding the first entry (either 1 or 2) to create the child's indices, i.e. the list indices $\left[1,L_t\right]$ and $\left[2,L_t\right]$ for the first and second child nodes, respectively.

Let us state and prove this observation.

\begin{thm}[Generating functions]
\label{thm:generating functions}
Let $g(a,b) = [ax+(1-a)]^{n-i}\cdot[bx+(1-b)]^{i-1}\cdot x.$ 
For any given $k\geq1$, and the list indices  $L_1, L_2,\dots, L_{2^k}$, 
\[ F_{L_t}(x) = g\left( \dfrac{t-1}{2^k}, \dfrac{t}{2^k} \right),\]
 for $t=1, 2, \dots, 2^k$
\end{thm}

We use the following lemma to find the sum of the terms in $F_L(x)$ recursively.
\begin{lem}
Assume $f(i) = A^{n-i}B^{i-1}x$ where $A= ax+(1-a)$ and $B = bx+(1-b)$. Then,
\[T_1(f) :=  \sum_{s=1}^i \dfrac{1}{2^{i-1}}\binom{i-1}{s-1}f(s) 
= A^{n-i}\left( \dfrac{A+B}{2} \right)^{i-1}x;\]
\[T_2(f) :=   \sum_{s=i}^n \dfrac{1}{2^{n-i}}\binom{n-i}{s-i}f(s) 
= \left( \dfrac{A+B}{2} \right)^{n-i}B^{i-1}x.\]
\end{lem}

\begin{proof}
First, it is easy to show that 
\[ \sum_{s=1}^i \dfrac{1}{2^{i-1}}\binom{i-1}{s-1}x^{s-1} 
= \left(\dfrac{x}{2}+\dfrac{1}{2}\right)^{i-1}  \]
and
\[  \sum_{s=i}^n \dfrac{1}{2^{n-i}}\binom{n-i}{s-i}x^{s-1}
=  \left(\dfrac{x}{2}+\dfrac{1}{2}\right)^{n-i}\cdot \left(x+0\right)^{i-1}.  \]

For the general case, 
\begin{align*}
\sum_{s=1}^i \dfrac{1}{2^{i-1}}\binom{i-1}{s-1}f(s)
&= \sum_{s=1}^i \dfrac{1}{2^{i-1}}\binom{i-1}{s-1}
A^{n-1}\left(\dfrac{B}{A}\right)^{s-1}x \\
&= A^{n-1}x \left( \dfrac{A+B}{2A} \right)^{i-1},   
\;\ \;\ \text{from the above identity} \\
&= A^{n-i}\left( \dfrac{A+B}{2} \right)^{i-1}x.
\end{align*}
This proves the first claim regarding $T_1(f)$. The 
second claim can be done similarly.
\end{proof}

We now return to the proof of Theorem \ref{thm:generating functions}.

\begin{proof}[Proof of Theorem \ref{thm:generating functions}]
We prove by induction on $k=|L|$.
For the base case, we verify that $F_{[\,\,]}(x)=x^i$, is indeed $g(0,1)$.
For the induction step, let $L' = [l_{k-1},..,l_{2},l_{1}]$ and $L=[l_{k},l_{k-1},...,l_{2},l_{1}]=[l_{k},L']$ . Assume $f := F_{L'}(x) 
=  g\left( \dfrac{t-1}{2^{k-1}}, \dfrac{t}{2^{k-1}} \right)$.
Then it follows from Lemma 3 that
$F_L(x) = T_1(f)= g\left( \dfrac{2t-2}{2^{k}}, \dfrac{2t-1}{2^{k}} \right)$ if $l_k=1$ 
and $F_L(x) = T_2(f)=g\left( \dfrac{2t-1}{2^{k}}, \dfrac{2t}{2^{k}} \right)$ if $l_k=2.$
\end{proof}

\textbf{Open Problem 1:} Find a combinatorial interpretation for
the generating function $F_L(x)$ of this card guessing problem. 

\section{Distribution of $p_L$ when $n$ is large}
\vspace{-4.5em}
\begin{flushright}
\begin{tikzpicture}
\node [anchor=west] at (1.2,.7) {\bf\scriptsize{47\%}};
\draw [fill=ColorGray] (0,0) rectangle (2,.5);
\draw [fill=ColorGreen] (0,0) rectangle (0.9,.5);
\end{tikzpicture}
\end{flushright}

{\bf Exact distribution of $p_L$}

Through the formula of $g(a,b)$ in Theorem \ref{thm:generating functions}, it is easy to see that $p_L$ is the sum of two independent binomial random variables and a constant. The following corollary formally states the distribution of  $p_L$, and eventually the exact distribution of $a_{i,j}^{(k)}$.


\begin{cor}
\label{cor:PoissonBinomial}
For any given $k\geq1$ and  $L_1, L_2,\dots, L_{2^k}$,  let $S_{L_t} $ be a random variable whose probability generating  function is given by $F_{L_t}(x)$ defined in Theorem \ref{thm:generating functions}.  Then,
\[
S_{L_t} \sim B_1+B_2+1,
\]
where $B_1\sim\text{Binomial}\left(n-i,\dfrac{t-1}{2^k}\right)$ and $B_2\sim\text{Binomial\ensuremath{\left(i-1,\dfrac{t}{2^k}\right)}}$ are independent. 
In other words, $p_{L_t}$ follows a Poisson Binomial distribution with success probabilities equal to
\[\underbrace{\dfrac{t-1}{2^k},...,\dfrac{t-1}{2^k}}_{n-i \text{ times}}, \underbrace{\dfrac{t}{2^k},...,\dfrac{t}{2^k}}_{i-1\text{ times}},1.\]

Hence, for fixed values of $n, i, k$,  the distribution $\{a_{i,j}^{(k)}: j=1,\dots, n\}$ is a mixture of $2^k$ components of Poisson Binomials.
\end{cor}

\begin{ex}
To better visualize the distribution, Figure \ref{fig:overlayHistK2} illustrates the p.m.f.  for the case $k=2$ shuffles when $i=15$ and $n=40$. 

\begin{figure}[ht]
\centering
\includegraphics[scale=0.23]{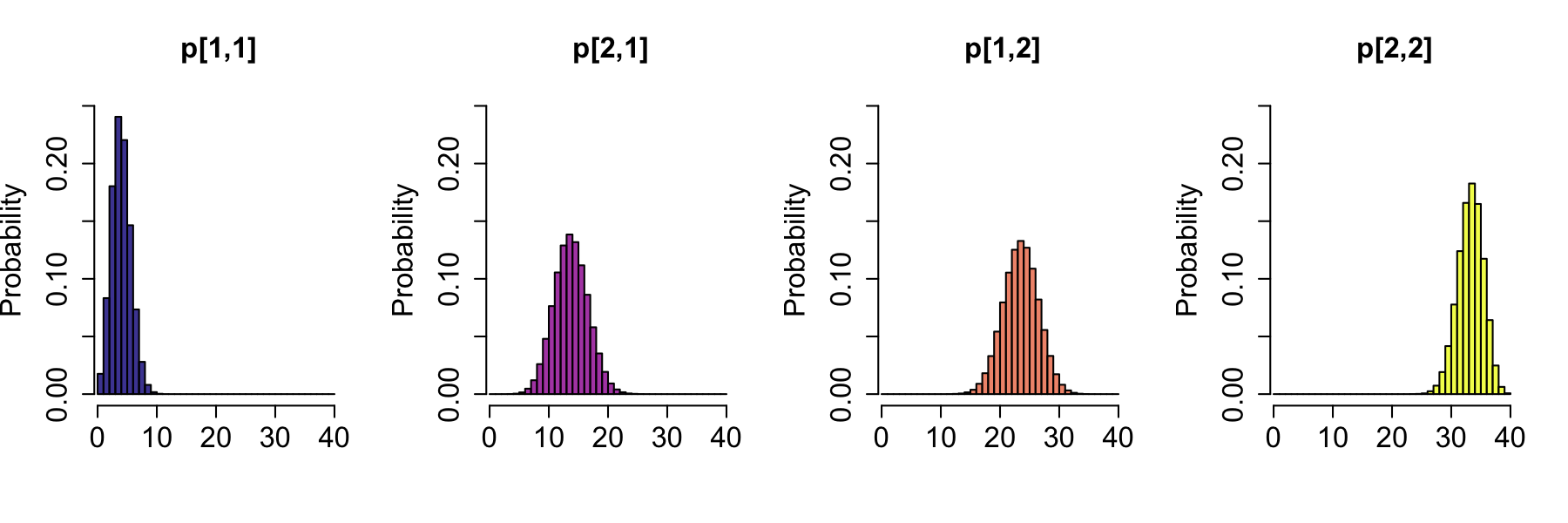}
\caption{\label{fig:overlayHistK2}Probability mass functions $p_{[1,1]}$, $p_{[2,1]}$, $p_{[1,2]}$, $p_{[2,2]}$ for $i=15$ and $n=40$}
\end{figure}

In this case, the distribution of each random variable is
\[S_{[1,1]}\sim\text{Binomial\ensuremath{\left(14,\dfrac{1}{4}\right)}}+1, \;\;\;\;\; S_{[2,1]}\sim \text{Binomial}\left(25,\dfrac{1}{4}\right)+\text{Binomial\ensuremath{\left(14,\dfrac{2}{4}\right)}}+1,\]
\[S_{[1,2]}\sim\text{Binomial}\left(25,\dfrac{2}{4}\right)+\text{Binomial\ensuremath{\left(14,\dfrac{3}{4}\right)}}+1, \;\;\;\; S_{[2,2]}\sim\text{Binomial}\left(25,\dfrac{3}{4}\right)+15.\]
\end{ex}


{\bf Lyapunov Central Limit Theorem }

The Lyapunov CLT generalizes the classical Lindeberg-Feller CLT. The
theorem provides a sufficient condition under which the sum of independent
(but not necessarily identical) random variables converges in distribution
to the standard normal distribution \cite{D}.

\begin{thm} [Lyapunov CLT]
 Let $X_{1},X_{2},\dots,$ be a sequence of
independent random variables such that $E\left(X_{n}\right)=\mu_{n}$
and $\text{Var}\left(X_{n}\right)=\sigma_{n}^{2}$ are both finite.
Define 
\begin{align*}
s_{n}^{2} & =\sum_{i=1}^{n}\sigma_{i}^{2}.
\end{align*}

If there exists $\delta>0$ such that 
\[
\lim_{n\rightarrow\infty}\frac{1}{s_{n}^{2+\delta}}\sum_{i=1}^{n}E\left(\left|X_{i}-\mu_{i}\right|^{2+\delta}\right)=0,\qquad\text{{(Lyapunov's\,condition)}}
\]
then 
\[
\frac{1}{s_{n}}\sum_{i=1}^{n}E\left(X_{i}-\mu_{i}\right)\stackrel{d}{\rightarrow}N\left(0,1\right).
\]
\end{thm}

The following example gives a sufficient condition for the sum of non-identical Bernoulli random variables to converge to a normal distribution.

\begin{ex}[Lyapunov's condition for Bernoulli random variables]
Let $X_{n}\sim\text{Bernoulli}\left(p_{n}\right)$ be a sequence of
independent Bernoulli random variables. Then, $\mu_{n}=p_{n}$ and
$\sigma_{n}^{2}=p_{n}\left(1-p_{n}\right)$. Since
\[
E\left(\left|X_{n}-\mu_{n}\right|^{3}\right)=p_{n}\left(1-p_{n}\right)^{3}+\left(1-p_{n}\right)p_{n}^{3}=\sigma_{n}^{2}\left(1-2p_{n}\left(1-p_{n}\right)\right)\leq\sigma_{n}^{2},
\]
it follows that (for $\delta=1$), 
\[
\frac{1}{s_{n}^{3}}\sum_{i=1}^{n}E\left(\left|X_{i}-\mu_{i}\right|^{3}\right)\leq\frac{1}{s_{n}^{3}}\sum_{i=1}^{n}\sigma_{n}^{2}=\frac{1}{s_{n}}.
\]
 
Thus, if $\lim_{n\rightarrow\infty}s_{n}=\infty$, or equivalently,
whenever
\[
\sum_{i=1}^{n}p_{i}\left(1-p_{i}\right)\rightarrow\infty,
\]
 the Lyapunov's condition holds for Bernoulli random variables.
 \end{ex}

\begin{cor}
\label{cor:LyapunovCLTforBinomials}
Let $0<p, q<1$ be fixed, and $T_{1}\sim\text{Binomial}\left(m,p\right)$
and $T_{2}\sim\text{Binomial\ensuremath{\left(n,q\right)}}$ be two
independent binomial random variables. Then, 
\[
\frac{\left(T_{1}+T_{2}\right)-\left(mp+nq\right)}{\sqrt{mp\left(1-p\right)+nq\left(1-q\right)}}\stackrel{d}{\rightarrow}N\left(0,1\right)
\]
whenever $m\rightarrow\infty$ or $n\rightarrow\infty$ or both hold.
\end{cor}
\begin{proof}
First note that $T_{1}+T_{2}$ has mean $mp+nq$ and variance $mp\left(1-p\right)+nq\left(1-q\right)$.
The result is immediate as we can write

\[
T_1+T_2\stackrel{d}{=}\underbrace{X_{1}+\dots+X_{m}}_{X_i\sim\text{Bernoulli}(p)}+\underbrace{X_{m+1}+\dots+X_{m+n}}_{X_i\sim\text{Bernoulli}(q)},
\]and the Lyapunov's condition for Bernoulli random variables, 
\[
\sum_{i=1}^{m+n}p_{i}\left(1-p_{i}\right)=mp+nq\rightarrow\infty,
\]
is satisfied whenever $m\rightarrow\infty$ or $n\rightarrow\infty$
or both hold.
\end{proof}

\vspace{-1em}
\begin{flushleft}
\shadowbox{\begin{minipage}[t]{1\columnwidth}%
{\bf \textsc{The trivial cases of $p_{[1,1,\dots,1]}$ and $p_{[2,2,\dots,2]}$}}\\

Corollaries \ref{cor:PoissonBinomial} and \ref{cor:LyapunovCLTforBinomials} tell us that as $n$ is sufficiently large, $p_{L_t}$ will approach a normal distribution, except for the trivial cases that might occur with $p_{[1,1,\dots,1]}$ and $p_{[2,2,\dots,2]}$. In particular, regardless of the magnitude of $n$, $p_{[1,1,\dots,1]}$ will remain a shifted binomial random variable, $\text{Binomial\ensuremath{\left(i-1,\dfrac{1}{2^k}\right)}}+1$, for small $i$, whereas $p_{[2,2,\dots,2]}$ remains $\text{Binomial}\left(n-i,1-\dfrac{1}{2^k}\right)+i$, for small $n-i$. 
\end{minipage}}
\par\end{flushleft}

\vspace{1em}
For the future reference,  the modes of $p_{[1,1,\dots,1]}$ and $p_{[2,2,\dots,2]}$ are located at 
\begin{equation}
\label{eq:mode1}
j_i = \lf \dfrac{i}{2^k} \rf + 1
\end{equation}
and
\begin{equation}
\label{eq:mode2}
j_i = \lf (n-i+1)\left(1-\dfrac{1}{2^k}\right) \rf + i,
\end{equation}
respectively, for $i=1,\dots, n$.

We shall hereafter call these two cases {\it``the trivial cases''}. In order not to interrupt the flow of the presentation, we will consider only the non-trivial cases of $p_{[1,1,\dots,1]}$ and $p_{[2,2,\dots,2]}$ (when the CLT applies) in detail, and defer the discussion of the trivial cases to the end of the section.  

We immediately obtain the following corollary (for the non-trivial cases).

\begin{cor}
\label{cor:CLT}
Let $k\geq1$ be fixed, and $t=1,2, 3, \dots, 2^k$. As $n\rightarrow\infty$,
\[
p_{L_t} \stackrel{d}{\rightarrow}N\left(\mu_{t},\sigma^2_t\right),
\]
where 
\[
\mu_{t}=(n-i)\left(\dfrac{t-1}{2^k}\right)+(i-1)\left(\dfrac{t}{2^k}\right)+1
\] 
and 
\[
\sigma^2_t= (n-i)\left(\dfrac{t-1}{2^k}\right)\left(1-\dfrac{t-1}{2^k}\right)+(i-1)\left(\dfrac{t}{2^k}\right)\left(1-\dfrac{t}{2^k}\right).
\]
\end{cor}

\begin{figure}[h!]
\centering
\includegraphics[scale=0.16]{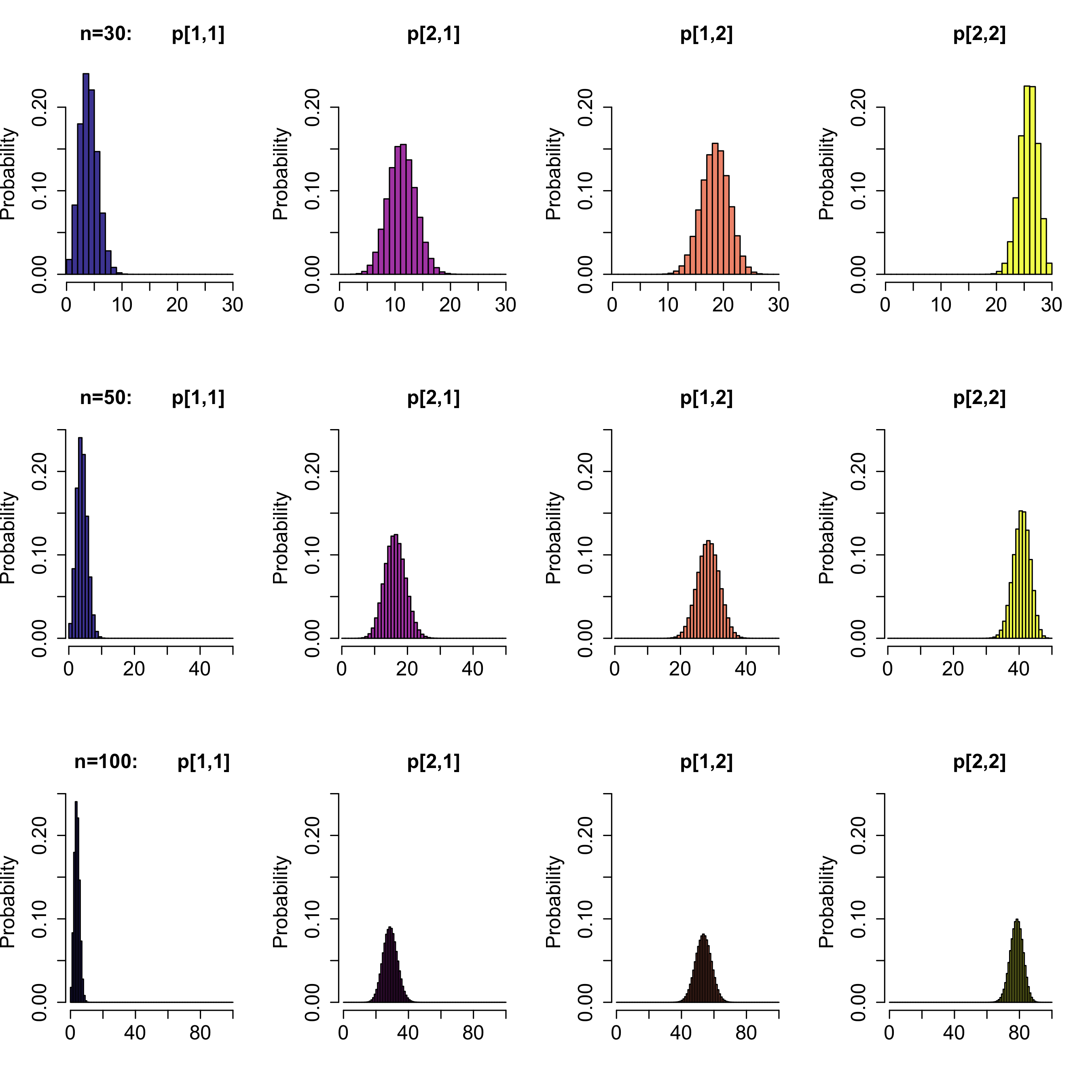}
\caption{\label{fig:overlayHistK2CLT} (Figure accompanying Remark \ref{rem:CLTapplicable}.) Probability mass functions $p_{[1,1]}$, $p_{[2,1]}$, $p_{[1,2]}$, $p_{[2,2]}$ for $i=15$, and $n$ varies from 30 (top row), 50 (middle row), and 100 (bottom row). Observe that regardless of the value of $n$, the distribution $p_{[1,1]}$ remains the same.}
\end{figure}

\begin{rem}
\label{rem:CLTapplicable}
As a practical note, whenever we come across the statement  {\it``as $n\rightarrow\infty$'}' or {\it``large enough $n$''}, the condition such as $n\geq N(k)$ can be used for the CLT to be applicable. This can be seen from Figure \ref{fig:overlayHistK2CLT}: $k=2$ shuffles where we fix $i=15$, and vary $n$ from 30, 50, and 100.  Here, $n=30$ seems pretty large already for the middle two distributions $p_{[2,1]}$ and $p_{[1,2]}$ to converge to a normal distribution. For the case when $t=1$, since $i=15$ is small, we have the trivial case here. The distribution $p_{[1,1]}$ which is independent of $n$ remains the same shifted Binomial regardless of the values of $n$ (see the first column of the figure). On the other hand, for the case when $t=2^k$, $n-i=n-15$ increases from $15, 35, 85$ as $n$ increases from $30, 50, 100$, and so the distribution of $p_{[2,2]}$ converges to a normal distribution in this case (see the last column of the figure).
\end{rem}

\section{Which peak is highest when $n$ is large?}
\vspace{-4.5em}
\begin{flushright}
\begin{tikzpicture}
\node [anchor=west] at (1.2,.7) {\bf\scriptsize{61\%}};
\draw [fill=ColorGray] (0,0) rectangle (2,.5);
\draw [fill=ColorGreen] (0,0) rectangle (1.2,.5);
\end{tikzpicture}
\end{flushright}

Our goal now is to compare the heights of the peak of each component $p_{L_t}$, $t=1, \dots, 2^k$  (for the non-trivial cases). To clearly visualize this, we overlay the plots of $p_{[1,1]}$, $p_{[2,1]}$, $p_{[1,2]}$, $p_{[2,2]}$ on the same axis in Figure \ref{fig:overlayHistK2Peak}. Here, we fix $n=100$ and vary $i$ from 10, 20, 50, and 91. We give an example of $i=91$ in order to emphasize the fact that the cases $i=10$ and $i=100-10+1=91$  are mirror images of each other along the vertical direction. The cases when $i=10$ and $20$ lead to the trivial case of $p_{[1,1]}$.

\begin{figure}[ht]
\centering
\includegraphics[scale=0.13]{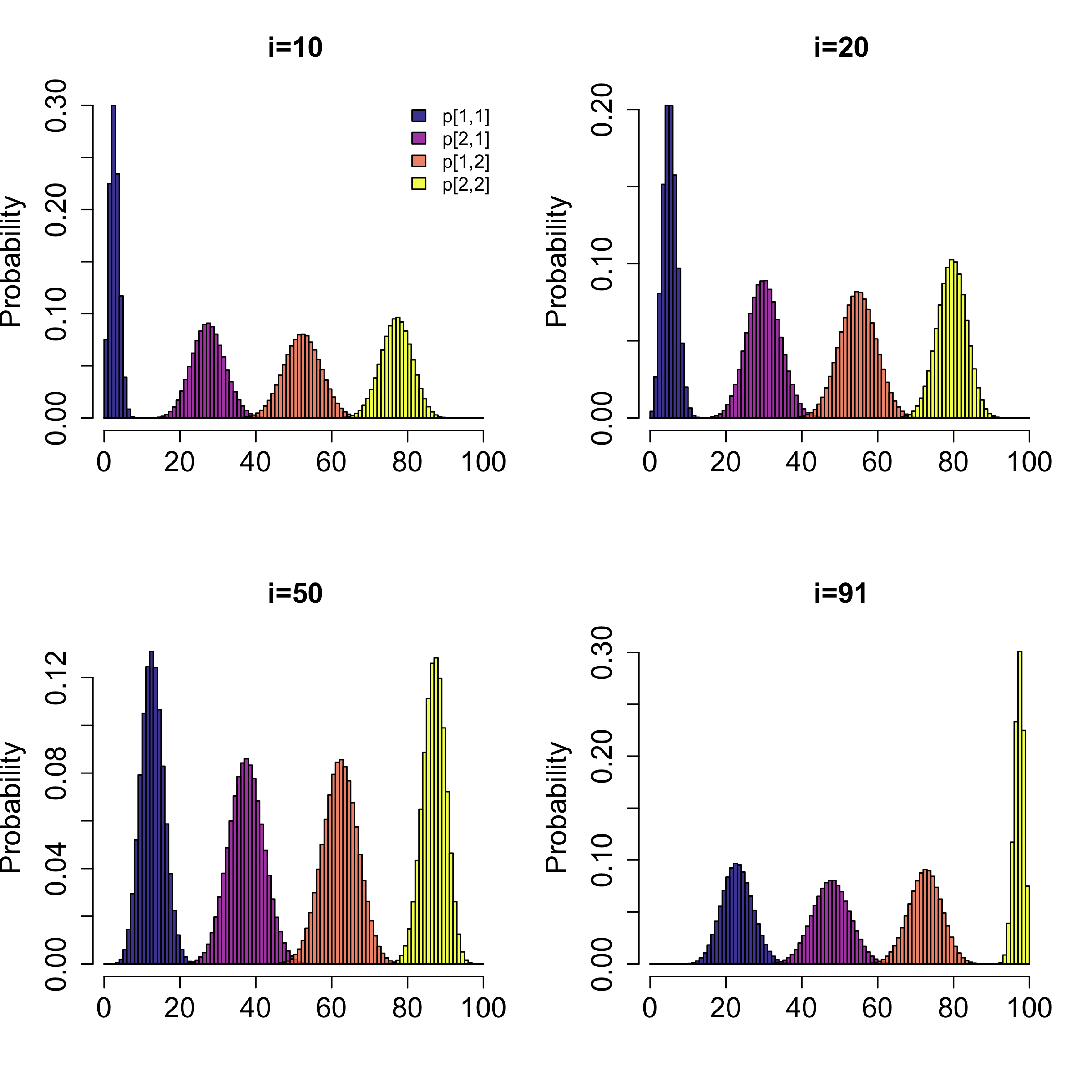}
\caption{\label{fig:overlayHistK2Peak}Probability mass functions $p_{[1,1]}$, $p_{[2,1]}$, $p_{[1,2]}$, $p_{[2,2]}$ for $n=100$, and $i$ varies from 10, 20, 50, and 91. Notice that $i=10$ and $i=91$ are mirror images along the vertical direction.}
\end{figure}

\textbf{The smaller the standard deviation, the higher the peak}

Recall the standard deviation $\sigma_t$ of each $p_{L_t}$ given in Corollary \ref{cor:CLT}, and the fact that the component $p_{L_t}$ with smallest standard deviation will have the highest peak. By comparing the values of $\sigma_t$, we obtain the following proposition which can be used to identify the component $p_{L^*}$ which achieves the highest peak among components $p_{L_t}$, $t=1, \dots, 2^k $.

\begin{prop}
\label{prop:highest peak} Let $n, i, k$ be fixed. For  $t=1, \dots, 2^k $, the height of the peak of $p_{L_t}$ (when $n$ is large) is approximately equal to $\dfrac{1}{\sigma_t\sqrt{2\pi}}$, where 
\begin{equation}
\sigma_t=  \sqrt{(n-i)\left(\dfrac{t-1}{2^k}\right)\left(1-\dfrac{t-1}{2^k}\right)+(i-1)\left(\dfrac{t}{2^k}\right)\left(1-\dfrac{t}{2^k}\right)}.
\end{equation}

Denote by $p_{L^*}$ the component whose peak is highest among $p_{L_t}$, $t=1, \dots, 2^k$. Then, 
\begin{equation}
p_{L^{*}}=\begin{cases}
\begin{array}{c}
p_{[1,1,\dots,1]}\\
p_{[2,2,\dots,2]}
\end{array} & \begin{array}{c}
\text{if i \ensuremath{\leq\left\lfloor \dfrac{(n+1)}{2}\right\rfloor }}\\
\text{otherwise.}
\end{array}\end{cases}
\end{equation}
\end{prop}

\begin{proof}
The height $\dfrac{1}{\sigma_t\sqrt{2\pi}}$ of the peak is due to the bell-shaped curve normal distribution. Using the fact that $f(p):=p(1-p)$ is unimodal and attains a minimum value at the boundary points, it is clear that if $i\leq\left\lfloor \dfrac{(n+1)}{2}\right\rfloor$, $\sigma_t$ is smallest when $t=1$. Hence, $p_{L^*}=p_{[1,1,\dots,1]}$ in this case. The result of the other case is justified in the same way.
\end{proof}

\section{Almost non-overlapping $p_{L_1}, p_{L_2}, \dots, p_{L_{2^k}}$}

\vspace{-4.5em}
\begin{flushright}
\begin{tikzpicture}
\node [anchor=west] at (1.2,.7) {\bf\scriptsize{79\%}};
\draw [fill=ColorGray] (0,0) rectangle (2,.5);
\draw [fill=ColorGreen] (0,0) rectangle (1.5,.5);
\end{tikzpicture}
\end{flushright}

Recall that our ultimate goal is to determine for each $i$, the index  $j^*_i$ that maximizes the mixture distribution   
\[
\label{eq:ak_decompose}
a_{i,j}^{(k)} = \dfrac{1}{2^k} \sum_{t=1}^{2^k} p_{L_t}(j).
\]

Thus far, we were able to show that, as $n$ is large, the mixture components $p_{L_t}$ are each normal distribution, and hence $a_{i,j}^{(k)}$ has the Gaussian mixture distribution. We have also identified the component $p_{L_t^*}$ whose peak is highest. Now, since the peaks of any two consecutive components  $p_{L_t}$ and $p_{L_{t+1}}$ are $\dfrac{n-1}{2^k}$ apart (easily verified by taking $\mu_{t+1}-\mu_t$), while the standard deviation of each $p_{L_t}$ is of order $\mathcal{O}(\sqrt{n})$, we can be sure that, as $n \to \infty$, the Gaussian components $p_{L_1},\dots, p_{L_{2^k}}$ are hardly overlapping.

\newpage
{\bf Maximum probablity $a_{i,j^*}^{(k)}$ is found!}

The fact that the components $p_{L_1}, p_{L_2}, \dots, p_{L_{2^k}}$ are (almost) non-overlapping (i.e. well-separated) when $n$ is large is very crucial and tremendously simplifies the problem of locating the maximizer $j^*_i$ of $ a_{i,j}^{(k)}$. In particular, the location of the highest peak $p_{L^*}$, identified previously in Proposition \ref{prop:highest peak}, is the maximizer we are seeking.

To see this, Figure \ref{fig:overlayHistK3} ($k=3$, $i=15$) illustrates the p.m.f of components $p_{L_1}, p_{L_2}, \dots, p_{L_{8}}$, and the superposition of the mixture distribution (the red line). We do not assign the weight of 1/8 to each component for the sake of clear visualization of the mixture. The location of the peak of the red line is the maximizer $j_i^*$. As can be seen from the figure,  when $n$ is small (top), the location of $j^*_i=3$ (red asterisk) does not  coincide with the location of the highest peak of $p_{L_1}$ (blue asterisk, at $j=2$). On the other hand, as $n$ becomes larger (bottom), the almost non-overlapping phenomenon occurs: while the peak of $p_{L_1}$ remains at the same height, the peaks of other components $p_{L_t}, t>1$, get lower and further apart  (remember the standard deviation is of order $\mathcal{O}(\sqrt{n})$).  Thus, the location of the maximizer $j^*_i$ coincides with the location of the peak of the component $p_{L_1}$ whenever $n$ is sufficiently large. 

\begin{rem}
We actually do not require $n$ to be that close to infinity. The components $p_{L_t}$ can still be overlapping, as long as there is no shift in the location of the maximum peak in the mixture model, and that is all we need. For example, we can see from the middle plot of Figure \ref{fig:overlayHistK3} that we already get the desired outcome when $n$ is as small as 50.
\end{rem}


\begin{figure}[h!]
\centering
\includegraphics[scale=0.23]{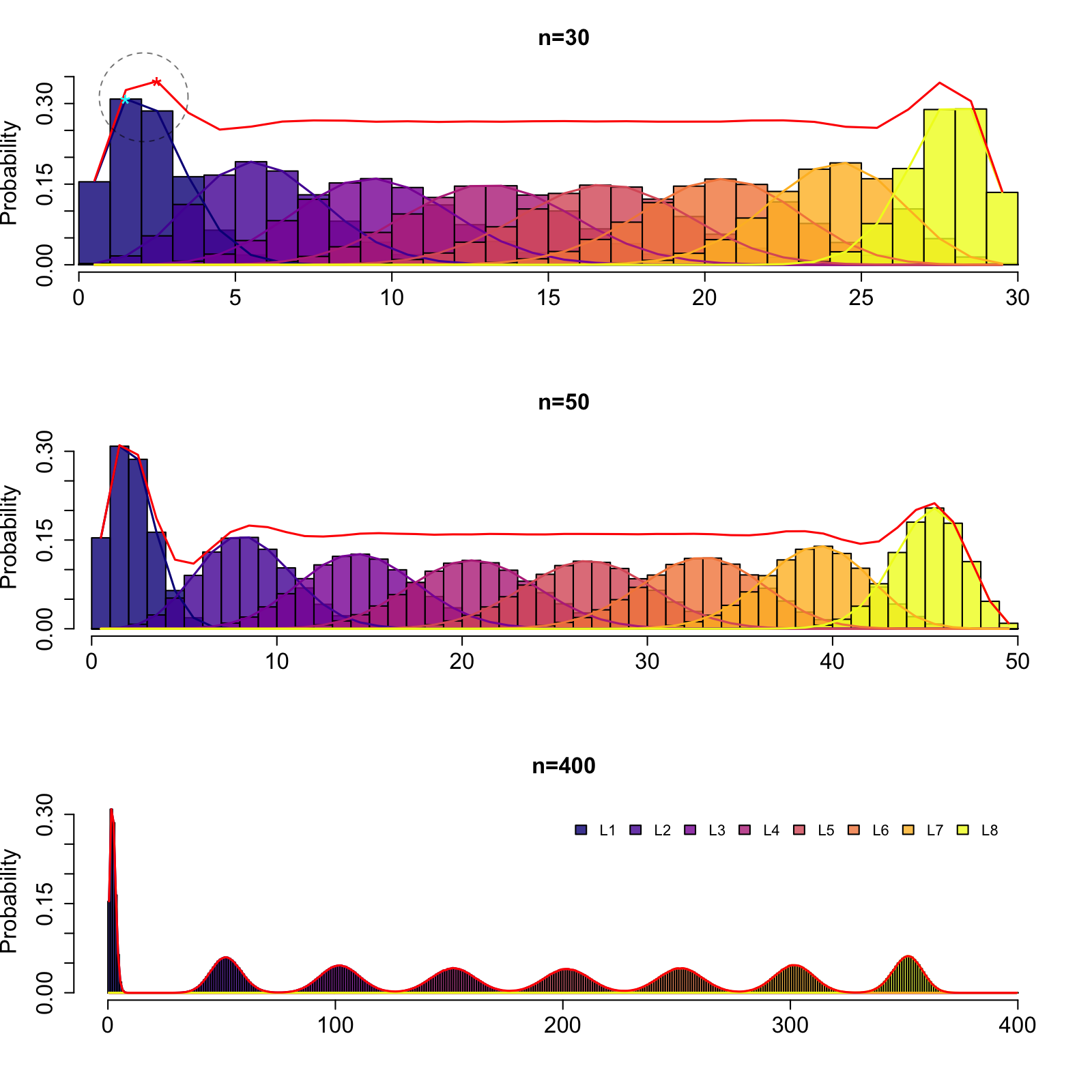}
\caption{\label{fig:overlayHistK3} Overlay probability mass functions $p_{L_1}$-$p_{L_8}$ for $k=3$ shuffles and $i=15$. The red line is the superposition of the mixture distribution (without the weight 1/8.) {\bf Top:} when $n$ is small, the location of $j^*_i$ does not necessary coincide with the location of the peak of $p_{L_1}$. {\bf Middle and Bottom:} when $n$ is large enough, the location of $j^*_i$  coincides with the location of the peak of $p_{L_1}$. }
\end{figure}

{\bf Revisit and settle the trivial cases}

Before we end this section, let us now settle the trivial cases. In fact, everything we have discussed so far applies to the trivial cases.  In particular the result regarding $p_{L^*}$ in Proposition \ref{prop:highest peak} still holds true. Consider for example the trivial case of $p_{[1,1,\dots,1]}$ when $i$ is small. Being a shifted Binomial with small and constant standard deviation $\sigma_1$, the peak of $p_{[1,1,\dots,1]}$ will indeed be highest.  Figure \ref{fig:overlayHistK3} (when $i=15$) also highlights the fact that $p_{[1,1,\dots,1]}$ does not depend on $n$. In particular, while the peaks of the other $p_{L_t}, t>1$ get lower as $n$ increases from $30, 50,400$, the peak of $p_{[1,1,\dots,1]}$ remains at the same height (and same location of course). 
Proceeding in the same manner,  the trivial case of $p_{[2,2,\dots,2]}$ when $n-i$ is small will achieve the highest peak, the result consistent with the second case of Proposition \ref{prop:highest peak}.

\section{Game solved!}
\vspace{-4.5em}
\begin{flushright}
\begin{tikzpicture}
\node [anchor=west] at (1.1,.7) {\bf\scriptsize{100\%}};
\draw [fill=ColorGray] (0,0) rectangle (2,.5);
\draw [fill=ColorGreen] (0,0) rectangle (2,.5);
\end{tikzpicture}
\end{flushright}

Now that we have done all the hard work, we are ready to wrap things up.  
The maximum probablity of $a_{i,j}^{(k)}$ is contributed
solely from the peak of $p_{[1,1,...,1]}$ if $i \leq 
\lf \dfrac{(n+1)}{2} \rf$
or $p_{[2,2,...,2]}$ otherwise.
Following immediately from the location of the mode of the shifted Binomial, $p_{[1,1,...,1]}$ and $p_{[2,2,...,2]}$ as given in Equations \ref{eq:mode1}-\ref{eq:mode2} and Proposition \ref{prop:highest peak},  the optimal strategy $\cal{G}^*$ can now be determined.


\vspace{-1em}
\begin{flushleft}
\shadowbox{\begin{minipage}[t]{1\columnwidth}%
{\bf \textsc{Optimal guessing strategy $\cal{G}^*$ for the $i$th position}}\\
\begin{itemize}
\item If $i \leq \lf \dfrac{(n+1)}{2}\rf$, the best guess is the number $j^*_i = \lf \dfrac{i}{2^k} \rf + 1$; 
\item If $i \geq \lc \dfrac{(n+1)}{2}\rc$, the best guess is the number $j^*_i = \lf (n-i+1)\left(1-\dfrac{1}{2^k}\right) \rf + i.$
\end{itemize}
\end{minipage}}
\par\end{flushleft}

That is, for a $k$-time shuffled deck of $n$ cards, one should guess the top half of the deck with sequence 
\[
\underbrace{1,\dots, 1}_{2^k-1\text{ times}}, \;\underbrace{2,\dots, 2}_{2^k\text{ times}}, \;\underbrace{3,\dots, 3}_{2^k\text{ times}}, \;\underbrace{4,\dots, 4}_{2^k\text{ times}}, \dots  
\]
and guess the bottom half in the reverse manner, i.e. 
\[
\dots, \underbrace{n-3,\dots, n-3}_{2^k\text{ times}}, \;\underbrace{n-2,\dots, n-2}_{2^k\text{ times}},\; \underbrace{n-1,\dots, n-1}_{2^k\text{ times}}, \;\underbrace{n,\dots, n}_{2^k-1\text{ times}}. 
\]

Using this strategy, we can calculate the expected number of correct guesses for the whole deck leading to the main result of this paper. It is worth noting that the expected value relies only on the standard deviation of the component $p_{[1,1,...,1]}$ of the first half of the deck, as we shall now see.

\shadowbox{\begin{minipage}[t]{1\columnwidth}%
\begin{thm}
\label{thm:main}
For an $n$-card deck, let $X^{\cal{G}^*}$ be the number of correct guesses (according to the optimal strategy $\cal{G}^*$) after $k$-riffle shuffles. 
Then,
\[ E\left[X^{\cal{G}^*}\right] = \dfrac{2\sqrt{n}}{\sqrt{(2^k-1)\pi}} +\mathcal{O}(1).\]
\end{thm}
\end{minipage}}

\begin{proof}
Recall Equation \ref{eq: EX}:
$ E\left[X^{\cal{G}^*}\right] = \sum_{i=1}^n a_{i,j^*}^{(k)} =  \dfrac{1}{2^k}\sum_{i=1}^n p_L(j^*_i), $
and the standard deviation of $p_{[1,1,...,1]}$:
$\sigma_1=  \sqrt{(i-1)\left(\dfrac{1}{2^k}\right)\left(1-\dfrac{1}{2^k}\right)}$.
By the symmetric property of the probability matrix $P^{(k)}$, we will consider only half of the deck and double the value of the summation. For an asymptotic value $n$, we have (from Proposition \ref{prop:highest peak}):
\begin{align*} 
E\left[X^{\cal{G}^*}\right] &= 2\sum_{i=2}^{\lf\frac{(n+1)}{2}\rf} \frac{1}{\sqrt{(i-1)(2^k-1)}\sqrt{2\pi}}+\mathcal{O}(1)\\
&= \dfrac{2}{\sqrt{(2^k-1)2\pi}}2\sqrt{\dfrac{n}{2}} +\mathcal{O}(1) \\
&= \dfrac{2\sqrt{n}}{\sqrt{(2^k-1)\pi}}+\mathcal{O}(1).
\end{align*}

Let us clarify the approximation error terms absorbed in $\mathcal{O}(1)$. Specifically, this includes (i) the error in normal approximations, when $i$ is small. The normal distribution gives a less accurate approximation, say, for the first $N(k)$ terms (finite number of terms for a given $k$) of the sum, after which the approximation becomes more and more accurate. The total error in this case is of order $\mathcal{O}(1)$;  (ii) the approximation errors due to the CLT (i.e. the normal approximation to the shifted Binomial). These approximation errors are asymptotically negligible compared to that in (i); and finally (iii) the errors of order $\mathcal{O}(1)$ due to approximation of a summation by an integral (in the second equality).
\end{proof}

From the theorem, notice that for each additional shuffle, the expected number of correct guesses will go down by roughly a factor of $\sqrt{2}$
(in a huge deck of $n$ cards, of course).
We also remark that for $k=1$, the result boils down to that of \cite{C}.

\vspace{1em}
\begin{center}
{\textsc{\large Game Complete!}}\\ \vspace{-0.6em}
\Huge{\Stars{5}}
\end{center}


\end{document}